\documentclass[11pt]{amsart}
\usepackage{graphicx,color}
\newtheorem{theorem}{Theorem}[section]
\newtheorem{lemma}[theorem]{Lemma}
\newtheorem{proposition}[theorem]{Proposition}

\theoremstyle{definition}
\newtheorem{definition}[theorem]{Definition}
\newtheorem{example}[theorem]{Example}

\theoremstyle{remark}
\newtheorem{remark}[theorem]{Remark}
\usepackage[english]{babel}
\usepackage[T1]{fontenc}
\usepackage{ifthen}
\usepackage{mathrsfs}
\usepackage{psfrag}
\numberwithin{equation}{section}

\DeclareMathOperator{\im}{Im}

\begin{document}
\title{Subprincipal Controlled Quasimodes and Spectral Instability}

\author{Pelle Brooke Borgeke}
\address{Linn\oe us university}
\curraddr{}
\email{pelle.borgeke@lnu.se }

\subjclass[2010]{Primary}


\date {6 January 2026}

\begin{abstract} Here we explore, in a series of articles, semiclassical quasimodes $u(h,b),$ approximative solutions $P(h)u(h)\sim 0$, depending on $0< h \leq 1$, and on  $b$, the subprincipal symbol. We study a pseudodifferential operator $P(x,hD_x;h^n B_{n\geq 1})$, with transversal intersections of bicharacteristics, where the principal symbol has double multiplicity, $p=dp=0$, in $(x_0,\xi_0,)\in \Omega \subset T^*(X), X\subset \mathbb{R}^{2n}$. Because of this fact, we instead study the subprincipal symbol $B_1(x,\xi)=b=\alpha(x,\xi)+i\beta(x,\xi),$ and we can conclude that we get transport equations depending on $b$ where sign changes for $\beta$ give approximative solutions with small support. These modes are used to estimate spectral instability, or the \emph{pseudospectrum}. We also investigate the possibility that we can factorize the model operator as $P(h)=h^2D_1D_2+ hB(x,hD_x)=h^2P_1P_2,$ in this way actually annihilating the subprincipal symbol, thus there is no $\beta$-condition. In a follow-up article, we examine different cases for more complex operators with tangential intersections of bicharacteristics, thereby generalizing the findings here.
\end{abstract}

\maketitle
\section{Introduction}
 This paper, in a series of articles, is an investigation of the conditions for the existence of \emph{semiclassical quasimodes}, $u(h,b)$, functions that are approximative solutions to $P(h)u(h,b)=0$. They depend of the small parameter $h\in (0,1],$ and here also of the subprincipal symbol $b$ for our \emph{model operator},  $P(h)=h^2D_1D_2+ hB(x,hD_x)$, (symbol, $p_1p_2+ b, \ b\in S^1$), which is a simplified micro local stand-in for the pseudodifferential operator, with as usual $D_x=-i\partial_x$ \begin{equation} P(x,hD_x;h^nB_{n\geq 1}) = P(x,hD_x) + hB_1(x,hD_x)+h^2B_2(x,hD_x)\ldots\end{equation}
This operator has in the small neighborhood $\Omega, p=dp=0$,  double multiplicity for the principal symbol, so we instead get \emph{subprincipal controlled} quasimodes. If these quasimodes $u(h,b)$ are present, we can expect spectral instability or pseudospectrum.
The novelty in this study, besides studying these types of quasimodes in spectral analysis, is that we encounter several different conditions for the quasimodes to appear, or be non-existent, for example, by factorization to $P_1(x,hD_x)P_2(x,hD_x;hB(x,hD_x))$, thus in fact annihilating the subprincipal control. We explore different classes of operators and have also developed and refined the proof method initially developed in the PDE theory of local solvability.
In this, the first article, we examine an operator class with \emph{transversal} intersections of bicharacteristics. In a second article, we consider a more complex class of operators, also with double characteristics for the principal symbol $p$ but with \emph{tangential} intersections of bicharacteristics.
The pseudospectra norm estimate is provided in the \emph{protocol} below. \footnote {We use here the layout to put statements and conditions in a \emph{protocol} (rules and conditions) to have all in one place and the quantifiers, for all $(x),$ there exist $[x]$ and unique $[\dot x]$ first. The \textbf{.}, dots, is used, but only in an obvious way, to avoid too many parentheses. We were inspired by [11] here. For more details, see Appendix C.}
\begin{align} (u(h,b))[h][P(h)][\dot\zeta][\dot\xi_1][\dot\xi_2](|R(h,\zeta)|=  \frac{\Vert u(h,b) \Vert}{\Vert P(h)u(h,b) \Vert_{L_2}} \longrightarrow \infty;\\ (u(h,b))\in C^{\infty}, ||(u(h,b))||=1, P(x,\xi)=\xi_1\xi_2, p=dp=0 \land (\xi_1=0,\xi_2=0), \\ 0<h \leq 1, h \rightarrow 0, \zeta=0; \ b(x,\xi)=\alpha+i\beta).\end{align}
We shall here find and prove the \emph{conditions} for this subprincipal control so that we can add these for the free variable $b(x,\xi)$ above (variables in (1.2)-(1.4) are “free” if they are not in the quantifiers).
We look for approximate solutions using $P(h)$ \footnote{Here we reduce the resolvent $R(h,\zeta)$ to $R(h,0)=R(h)$. If we study $P(h)=(P(h)+\zeta)$ as an operator, adding a constant to $P(h)$ does not change the results by more than a translation. The $\zeta$ in $(P(h)u(h)-\zeta)^{-1}$ can thus be ignored (subtracted).} on the ansatz $v_h$ \begin{align} [P(h)][v_h][B(x,hD_x)](P(h)v_h=h^2D_1D_2v_h +h B_0(x,hD_x)v_h\sim 0 ;\\ B(x,hD_x)=A_1hD_1+A_2hD_2+R(x) \in \Psi^1 \land A_1,A_2,R(x) \in \Psi^0 ,\\ C^\infty \ni v_h(b)= e^{ig/h^{\alpha}}a(x,b); \ g=x_2\xi_2) \ . \end{align}
Here, “$\alpha$” is one of the parameters used in the exponent of $h$ to govern the system of transport equations. In the main theorem, we find that the condition to get the quasimodes in this setting is: The \emph{imaginary part} of the subprincipal \emph{symbol}, $\beta(x,\xi)$, must \emph{change sign} on a \emph{limit bicharacteristic}. \footnote{As we have $dp=0$, the solution to the Hamiltonian equations $(\partial_{\xi}p,-\partial_{x}p)$ will be zero for $\xi_1=\xi_2=0$ and the bicharacteristic there will just be a point. Because of this, we approach the point instead of a limit bicharacteristic. For more details, see Appendix A or [4].} We shall then say that the operator (1.1) has a \emph{transversal} subprincipal normal form $\beta$-condition in $\Omega.$ In Section 4, we study an operator with the same conditions, but we cannot construct the quasimodes due to the factorization of the operator.
The material is split as follows: In Section 2, we introduce the proof method and discuss several points. In Section 3, we prove the theorem. Section 4 is the example concerning factorization, mentioned above. After that, we have put appendices to make the article self-contained. There, you first find some preliminaries and the necessary redefinitions of the geometry to fit our study of spectral problems, followed by a second appendix (B) with a short introduction to pseudospectrum, and in (C), comments on the notation and terminology used in this series of articles.
\section{The proof method}
In the proof, we use a system of transport equations which was used by H\"ormander [6] in his study of the question of solvability for linear differential operators, and also adopted by Dencker, f.x. in [4] and Mendoza and Uhlmann in [8], there also investigating operators with double multiplicity, and others who work with similar problems.
Here, we use parts of this framework in another area, semiclassical analysis to find quasimodes, so that $Pu(h) \sim 0$ can be connected to the question of pseudospectrum. We adjust and simplify the proof method to ensure it works smoothly with all the parameters and details needed here. Admittedly, it will be rather technical anyway, but we shall already discuss some of the details in the proof. We also prove several lemmas that will simplify the next section by dividing the proof into smaller pieces. In this way, we get a shorter, simpler main proof that extends to just over three pages.
We shall use a semiclassical version of the usual asymptotic expansion from Duistermaat [5] (conjugated by $e^{-i g/h^{\alpha}}$) to calculate the model operator's action on the ansatz $v_h(b)=e^{i g/h^{\alpha}} a(x)$
\begin{align}e^{-i g/h^{\alpha}} (P(h)e^{i g/h^{\alpha}} a )(x) = p(x, h\partial_x  g/h^{\alpha}) \\ +  h b(x,h^{1-\alpha} \partial_x  g)a  + h \sum_j \frac{\partial}{\partial \xi_j}p(x,h^{1-\alpha} \partial_x  g)D_j a \\ + \frac{1}{2} h^2 \sum_{ij} \frac{\partial^2 }{\partial \xi_i \partial \xi_j}p(x, h^{1-\alpha}\partial_x  g)D_iD_ja + \mathcal{O}(h^3).
\end{align} For our operator class, as the principal symbol is just $p=\xi_1 \xi_2,$ the $\emph{eikonal}$ \footnote{The word eikonal comes from the Greek word $\epsilon \iota \kappa \omega \nu$ which means image or just icon.} equation, $ p(x,h^{1-\alpha}\partial_x g)a=0$, if we take $g(x)$ dependent on just one of the $x$ variables so $g(x)=\xi_2 x_2$. As a result, the subprincipal symbol gets into the first transport equation.
We work with principal symbols, $p$, that microlocally can be factored $p=p_1p_2$  in the neighborhood $\Omega$, which we study. The reduction to the normal form is taken from [1], where it is done with $h=1$. This is not a restriction because we can always rescale to this case via a change of variables, as shown on page 57 in [12]. \begin{equation} \tilde{x}:=h^{-1/2}x, \  \tilde{\xi}:=h^{-1/2}\xi, \ a_h(\tilde{x},\tilde{\xi}):=a(x, \xi).\end{equation} After that, we use semiclassical Fourier integral operators instead. We first introduce three parameters. \begin{lemma} The parameters $\alpha, \beta$ and $\gamma$ in our system of transport equations form a partition of the unity of the semiclassical parameter $1<h \leq 1$ so that $h=h^{\alpha + \beta + \gamma}$. They are real numbers, used in the exponent of the semiclassical variable to balance the system of equations and also adjust the subprincipal symbol, satisfying \begin{equation} [\alpha, \beta, \gamma](\alpha + \beta + \gamma = 1, 0 <\beta < \gamma \leq \alpha < 1;  \    \alpha, \beta, \gamma \in \mathbb{R}(0,1)). \end{equation} which is a plane in $\mathbb{R}^3$ with vertices in $(\alpha, \beta, \gamma)= (1,0,0), (0,1,0), (0,0,1)$.
\end{lemma}
\begin{proof}
We can write $(\alpha, \beta, \gamma)= (1,0,0) + s(-1,0,1) + t(0,-1,1),$ which give
\begin{equation}\left\{  
\begin{array}{ll}
1.\ \alpha = 1  -s \\
2.\ \beta =    -t\\
3.\ \gamma =  s+ t \\
4.\ \alpha + \beta + \gamma=1
\end{array}\right.
\end{equation}
If we take $s,t$ as positive numbers, we find that ($ 0<\beta <\gamma \leq \alpha <1).$
The equation (4.) comes from semiclassical quantization in $h$ ($\xi\mapsto h\xi \mapsto hD_{x}$) and derivation of $g(x)=x_2h\xi_2/h^{\alpha}$ which gives $h^{1 -\alpha}$ in the exponent. A change of coordinates contributes with $\gamma$ for length coordinates and $-\gamma$ for derivation, and in that case we get $h^{1 -\alpha-\gamma}$. We define $\beta:=1 -\alpha-\gamma$, and this simplifies to $\beta=1-2\gamma$, if we use the equality in $\gamma \leq \alpha.$ This can be useful in some context; it specifies the parameters by $1=2\gamma+\beta$ and then we must choose $\gamma=2j_{=2|3}\beta$, which means $$(\beta=1/(2j_{=2|3}+1)) \Leftrightarrow (\beta=1/5 \   |      \ \beta=1/7).$$
\end{proof} The parameter “$\gamma$” changes coordinates (except for $x_1=t$ and $D_1$) so that $b(t,h^{\gamma} x,h^{\beta}\xi_2) \mapsto b(t,h^{\beta}\xi_2) +\mathcal{O}(h^{\gamma})$ in $\Sigma_2(z),$ the $t=x_1$ line where the sign change is supposed to be. The value of $\gamma$ must also be enough to control the remainder from getting unbounded. This means, and we take the usual association with $S^n$ as symbolspace and $\Psi^n$ as operatorspace, \begin{align}
(S^n(\mathbb{R}^n)\ni(x_{j>0},\xi_{j>0}) \mapsto (x_{j>0}, hD_{j>0})\in \Psi^n)\\ \land
(S^1\in \xi_{j>0} \mapsto h\xi_{j>0} \in \Psi^1).\end{align} but we make the change of coordinate in the operator space by $h^{\gamma}$ on $(x_{j>1}, \xi_{j>1}),$ \begin{equation} \Psi^n(\mathbb{R}^n)\ni (x_{j>1}, hD_{j>1})\mapsto(h^{\gamma}x_{j>1}, h^{1-\gamma}D_{j>1})\in \Psi^n(\mathbb{R}^n). \end{equation}
This technique of scaling or rescaling is standard in semiclassical analysis with $h$ as a trademark, and this letter is scaled in different ways for different reasons. We have already rescaled to $h=1$ to use classical results; for more on this, see [12].
The following WKB-form \begin{equation} (v_h(x)=e^{ig/h^{\alpha}}a(x)) \land (g(x) = x_2 \xi_2) \end{equation} is the ansatz $v_h$, which is acted on by the model operator.
We shall prove lemmas and give examples that connect to these parameters, results that we can also use in the following article. We will later see that we get, after using Duistermaat´s formula with the coordinates ($(t,x)\ | \ (x_1,x'))$ \begin{equation} h^{2-\alpha}\xi_2 D_1 a +hb(x_1,hx',h^{1-\alpha} \xi_2)a + h^{2}D_1D_2 a + \mathcal{O}(h^2) \end{equation} and the remainder term from the expansion $\mathcal{O}(h^2).$ The idea is to get an equation with $D_1a$ and $ba$, the first and second term, as the first transport equation, and solve it by integrating factor, leaving the other terms to the following equations. This will give us $a(x,b)$ as an exponential function depending on $b$. If we scale everywhere but $x_1$ and $D_1$, using Taylor in the second step for $x'$, with remainder $\mathcal{O}(h^{1+\gamma})$, and the definition of $\beta$, we find \begin{align}
h^{1+\beta}\xi_2 D_1 a +hb(x_1,h^{\beta} \xi_2)a + h^{2-\gamma} D_1D_2 a + \mathcal{O}(h^{1+\gamma}) + \mathcal{O}(h^{2-\gamma}). \end{align}
We can now factor out $h^{1+\beta}$, freeing $\xi_2 D_1$ for the first transport equation, without making the remainder terms unbounded, which now becomes $\mathcal{O}(h^{\gamma-\beta})$ and $h^{1+\beta}\mathcal{O}(h^{1-\gamma-\beta})$. \footnote{In this way, the parameters work as balancing factors, as e.g., in a chemical reaction in stoichiometry. This is achieved by scaling the number of involved molecules for example, if we have $X, Y, Z$ and $W$ with integers $k, l, m$ and $n$ we must choose the integers (parameters, $k,l,m$ and $n$) so that $k X + l Y  \longrightarrow mZ + nW$ gives a conservation of mass.} We shall simplify this further by using only $\beta$ as the unit parameter. This comes naturally as the quotient $\gamma/\beta>1$ is important here. We know that $\beta$ divides $\gamma,$ and if $\gamma=2\beta$ we obtain $\alpha=1-3\beta$. In all $1-3\beta+\beta +2\beta=1$. Also $\alpha/\beta=(1-3\beta)/\beta= 1/\beta-3=x-3$, where $x\in \mathbb{N}$ is the modulo factor. We shall later see that this is the minimum information the system needs.\footnote{Of course, we can use explicit numbers as $2/5+1/5+2/5$, which is the lowest explicit denominator that function, but eventually these numbers are all of the form $1-(j+1)\beta+\beta+j\beta=1, j=2,3.$}
For the general remainder term (also $\mathcal{O}_{\mathcal{G}})$ we get, when scaling every term but $D_1$, with $\gamma$ so $h^{\lambda}D_1^{\lambda}, h^{\mu-\mu \gamma}D_{x'}^{\mu}$ and $h^{\kappa-\kappa \gamma-\kappa \alpha} \xi_2^{\kappa}$ we find after simplifications, factorization and, using $\beta=1-\alpha-\gamma$ and $(| :=$ XOR) \begin{align} \mathcal{O}(h^{\kappa\beta+\lambda+ (1-\gamma)\mu}) \\ =h^{1+j_{=1|2}\beta}\mathcal{O}(h^{m + (\kappa-j_{=1|2}-(\gamma/\beta)\mu)\beta }) \\ =h^{1+1|2\beta}\mathcal{O}(h^{m + ((\kappa -1|2)-2|3\mu)\beta)} \end{align} with $m=\lambda+ \mu -1.$
We collect this result in a lemma that we use here and in the following article in this series.
\begin{lemma} For the remainder term for our two classes of operators, transversal and tangential intersections of bicharacteristics, we use in (2.16), the lowest numbers $1 \land 2$, for the first case, and $2 \land 3$, for the second case, with $m=\lambda+ \mu-1$. \begin{align} h^{1+1|2\beta}\mathcal{O}(h^{m + ((\kappa-1|2)-(2|3\mu))\beta).} \end{align} \end{lemma} \begin{proof} We must have $\gamma > \beta$ for a limited remainder. Moreover $\gamma - \beta>0.$  We have the following members in the parameter set $(0, \beta, \gamma, \alpha, 1)$ and as we know the order, this gives $\gamma - \beta=\beta $ so $\gamma=2\beta.$ It follows that if we factor out $h^{1+2\beta}$ we must have $\gamma =3 \beta$ and so on. For $\alpha$ we get the implicit form $\alpha=1-j_{=3|4}\beta,$ so we could use both. \footnote{We use "|" as the exclusive or, XOR(0110), and for the Sheffers stroke (NAND)(0111) we use $\uparrow$, which fits with (NOR)(0001) $\downarrow$ both being of the same character. The reason for this is that XOR is useful in communicating in mathematics, and we think it is presented best with $a|b:=(a\lor b) \land \lnot(a\land b).$} \end{proof} \begin{example}We can test this on $h^{\alpha} D_1D_2a$ from the expansion and with $j=1$ we get $$\mathcal{O}(D_1D_2a)=\mathcal{O}(h^{ (1+1-1)+(0-1-2)\beta})=\mathcal{O}(h^{1-3\beta}).$$ And for $\xi_2 D_1$ we get  $\xi_2D_1=\mathcal{O}(1)$, but for the first transport equation we get $D_1=\mathcal{O}(h^{-\beta}),$ because of the factorization. \end{example}
\begin{example} From another transport equation system, we factor out $1+2\beta$. We then have $\gamma=3\beta$ and $\alpha=1-4\beta.$ We can then decide the first transport equation from
$$h^{1+2\beta}\mathcal{O}(h^{m + ((\kappa-2)-3\mu)\beta)}.$$ And we get for $$(\xi^j_2D_1a)= h^{1+2\beta} \mathcal{O}(h^{(\kappa-2)\beta)} $$
which shall be bounded, as in the example above, $$h^{1+2\beta}\mathcal{O}(h^{((\kappa-2)\beta)})=\mathcal{O}(1) $$ and then $\kappa=2.$ We conclude that we get the following \begin{align} h^{1+2\beta}\xi^2_2D_1a=hb(t)a \land a(t)=e^{-i \int b(t)dt/\xi^2_2 h^{2\beta}}.\end{align} \end{example} We now sum this result up concerning the parameters
\begin{lemma}
The modulo factor $x$ in $1=x\beta$ is redundant in the transport equation system. We shall only need
\begin{equation}\left\{  
\begin{array}{ll}
1.\ \alpha = 1- (j_{=1|2}+2)\beta  \\
2.\ \beta =  \beta \\
3.\ \gamma =  (j_{=1|2}+1)\beta \\
4.\ \alpha + \beta + \gamma \\
=1- (j_{=1|2}+2)\beta+\beta+(j_{=1|2}+1)\beta=1
\end{array}\right.
\end{equation}
Possible modulo factors are $5$ in the transversal case and $7$ in the tangential case, as noted in the proof of Lemma 2.1. \end{lemma}
We also need a lemma concerning the transformation of the ansatz for the solution $v_h(x)$ \begin{equation}v_h(x)=e^{ig/h^{\alpha}}\sum_{j = 0}^N a_j h^{j \beta}, \quad a_j \in C^{\infty}, \end{equation} to $u(h)$ with $\Vert u(h) \Vert =1.$ This result will give both an upper and a lower estimate of $||v_h||$, and we write $\mathscr{B}(t)= \int_0^t b(t)dt/\xi_2,$ the coordinates are $(t,x)=(x_1,x')$.
The lemma is from [6], but is presented here in semiclassical form.
\begin{lemma} We assume in the upper estimate for \begin{equation}v_h(x,b)=e^{ig/h^{\alpha}}  \sum_{j = 0}^N a_j(x,b) h^{j \beta}, \quad a_j \in C^{\infty} \end{equation} that we have \begin{align}(\phi_j)(h)[v_h][a_j][\mathscr{B}(t)]\bigl( ||v_h || \le C,\ a_j(t,x) = \phi_j(t,x)e^{-i\mathscr{B}(t)/h^{\beta}};\\ \mathscr{B}(t)\in C^{\infty},\im \mathscr{B} \le 0, \phi_j(t,x) \in C_0^{\infty}(\mathbb{R} \times \mathbb{R}^n), 0<h \ll 1 \bigr).\end{align}
For the lower estimate, we have \begin{align}[v_h] (\phi_0), [\mathscr{B}(0)] [\beta, c] [h] (||v_h|| \ge ch^{(n\alpha+\beta)/2}; \\ \phi_0(0,0) \not=0,  \mathscr{B}(0) = 0,\ \beta, c>0,  0<h \ll 1).\end{align}
The estimates are uniform if we have uniform bounds on $\mathscr{B}$ and  $(j)(\phi_j)$.
\end{lemma}
\begin{proof}The upper estimate is immediate as we have an exponent with a non-positive real part.
We take for the lower estimate $\psi_h(t,x)= \psi(h^{-\beta}t, h^{-\alpha}x), \psi \in C^{\infty}0,$ and consider $$\langle \psi_h, v_h \rangle = \int{\mathbb{R}^n} \int{\mathbb{R}} \psi(h^{-\beta}t, h^{-\alpha}x)e^{ix \cdot \xi h^{-\alpha}} \sum_j \phi_j(t, x)  e^{-i\mathscr{B}(t)/h^{\beta} }h^{j \beta}dtdx.$$ We make a change of coordinates $(t,x) \mapsto (h^{\beta}t, h^{\alpha}x)$ and find
\begin{align} h^{-(n \alpha+\beta)}\langle \psi_h, v_h\rangle =  \iint  \psi(t,x)e^{ix \cdot \xi} \sum_j \phi_j(h^{\beta}t,h^{\alpha}x)  e^{-i\mathscr{B}(h^{\beta}t)/h^{\beta} }h^{j \beta}dtdx \\   \nonumber \to  \iint  \phi_0(0,0) e^{-i\mathscr{B}'(0)t} \psi(t,x) e^{ix \cdot \xi} dt dx = c_0 \quad \end{align} where we used Taylor so $\mathscr{B}(h^{\beta}t) \sim h^{\beta}\mathscr{B}'(0) t$ and dominated convergence as $ h \to 0$ since $||\psi_h||^2 = h^{ (n \alpha+\beta)/2} ||\psi||^2 $ and the coefficient $c_0 \not= 0$ for suitable choice for $\psi$.
For a small enough $h$, we particularly have $$h^{(n \alpha+\beta)/2} \frac{|c_0|}{2}
\le h^{-(n \alpha+\beta)/2}| \langle \psi_h, v_h\rangle |.$$
In view of the limit, Cauchy-Schwarz and scaling give for sufficiently small $h$ that $$h^{(n \alpha+\beta)/2} \frac{|c_0|}{2} \le h^{-(n \alpha+\beta)/2} | \langle \psi_h, v_h\rangle |
\le |\psi| |v_h|,$$ from which the estimate follows.\end{proof}
Below, we review the vocabulary of the theorem.
The \emph{closure of the range of p} is \begin{equation} \Sigma (P(h))= \Sigma(p) =\overline{\text{Ran}} (p) \end{equation}
The \emph{stationary points} are \begin{equation}\Sigma_2(P(h)) = { d_{\xi}p(x_0, \xi_0) = 0.}
\end{equation}
The $\emph{stationary point level set}, \Sigma_2(z) =  p^{-1}(z) \cap \Sigma_2(P(h))$   \begin{equation} [x_0] [\xi_0](p(x_0, \xi_0)= \zeta,d_{\xi}p(x_0, \xi_0) = 0,\zeta \in \Sigma (P(h))).\end{equation}
To prove our estimates for the resolvent norm, we will use the following.\begin{definition}
We consider $v=v_h(x,b)$ a $h$-dependent function of $x\land b$  and the  operator norm
\begin{equation}\Vert A \Vert = \sup_{v \not=0} \frac{\Vert Av \Vert}{\Vert v \Vert} \end{equation}
The \emph{semiclassical limit} of the resolvent norm in $L^2$ is defined by
\begin{align}\Vert R(h) \Vert = \Vert (P(h))^{-1} \Vert = \sup_v \frac{\Vert (P(h))^{-1}v \Vert}{\Vert v \Vert} = \sup_{v= P(h)u} \frac{\Vert (P(h))^{-1}v \Vert}{\Vert v \Vert}  \\ = \sup_{u} \frac{\Vert u \Vert}{\Vert P(h)u \Vert} \ge  \frac{\Vert u(h) \Vert}{\Vert P(h)u(h) \Vert} \to \infty ; \\
[u(h)][h](\frac{\Vert u(h) \Vert}{\Vert P(h)u(h) \Vert} \to \infty, ||P(h)u(h)|| \sim 0, \Vert u(h) \Vert \sim 1, h \rightarrow 0).
\end{align} \end{definition}
The limit is the method for determining whether we have a pseudospectrum, with no modes or just eigenmodes; this expression does not exist. \begin{remark} We shall also note that the solution to $\gamma'=(\dot{x}(t),\dot{\xi}(t)) =(\partial_{\xi}p,-\partial_x p)$ which give us the bicharacteristics is in this case $(\partial_{\xi}p_1,-\partial_x p_1)=(1,0)$ so that $(\gamma= t+c,d)$, but as we consider the null-bicharacteristics $p(\gamma(t))=0$ this is just lines in the $t$ direction$(x_1)$. We get the same with $p_2$, but now the lines lie in the $x_2$ direction, so the intersection is transversal; in fact, these lines generate planes. In our case we have double multiplicity at $\Sigma_2(P(h)) = {d_{\xi}p(x_0, \xi_0) = 0 }$, so the bicharacteristics can degenerate to a point. This means that we here consider the set $(\Gamma_j)$ of $\emph{limit bicharteristics}$ from outside $\Sigma_2(P(h))$ on the symplectic foliation of the open set $\Omega \supset \Sigma_2(P(h))$. When the bicharacteristics are parameterized with respect to arc length and are uniformly bounded in $C^{\infty},$ then $\Gamma_j$ is a closed and bounded set, so it is  sequentially compact and by Bolzano-Weierstrass theorem it has a subsequence $\Gamma_{j_k}$ that converge to a smooth curve $\Gamma$ (possibly a point) in $\Sigma_2(P(h))$. For more of this, consult [4]. \end{remark}
To continue, we use the semiclassical \emph{injectivity} pseudospectrum from [9], which is suited to our theoretical study. \begin{definition} For $P(h), {0<h \le 1},$ the semiclassical family of operators on $ L^2(\mathbb{R}^n) $, the spectrum is defined by $\sigma_{s}(P(h)) = {z \in \mathbb{C}: P(h)-z \ \text{is not invertible}}$.
The $\emph{semiclassical pseudospectrum}$ is denoted $\sigma^{N}_{scps}(P(h))$ and defined by
\begin{equation}(h_0)[P(h)][h][\zeta](\Vert (P(h)-\zeta)^{-1} \Vert \ge Ch^{-N}, 0< h < h_0, C>0, N \ge 0, \zeta \in \mathbb{C}).\end{equation}
We may change the statement to $(\Vert (P_h-\zeta)u \Vert < Ch^{N})$ or\ $(\Vert P_hu \Vert < Ch^{N} )$ with $||u(h)|| = 1$ to get the $\emph{semiclassical injectivity pseudospectrum},$ $\sigma^{N}_{scips}(P(h))$ of the family $P(h)_{0<h \le1}.$ If the inequalities are true for any $N$, we say that we have pseudospectra of infinite index, and we write $\zeta\in\sigma^{\infty}_{scps}(P(h))$ and $\zeta\in\sigma^{\infty}_{scips}(P(h)).$ \end{definition} \begin{remark} It can be proved that $\sigma^{N}_{scips}(P(h)) \subset \sigma^{N}_{scps}(P(h))$, see [9] and Apendix B. As we have the same resolvent-norms for $P(h)$ and $P^*(h)$ and \begin{equation} \sigma^{N}_{scips}(P(h)) \cup \sigma^{N}_{scips}(P^(h)) = \sigma^{N}_{scs}(P(h)) \end{equation} and since we work with conditions that hold for both $P(h)$ and $P^*(h)$ we consider $\sigma^{N}_{scips}(P(h)) = \sigma^{N}_{scps}(P(h))$, or $\sigma^{\infty}_{scips}(P(h)) = \sigma^{\infty}_{scps}(P(h)).$
We also observe that if $R(h)$ is not bijective, we get that $z \in \sigma_s(P(h))$ in the definition, in which case we define $||R(h)|| = \infty,$ so here we study the case where the resolvent is bijective. \end{remark}
We can adjust the Duistermaat’s formula for our needs with $g=x_2 \xi_2 $ and after a change of coordinates ($hD_j\mapsto h^{1-\gamma}D_j$; $\xi_2\mapsto  h^{-\gamma}\xi_2$) and $(x_j \mapsto h^{\gamma}x_j, j>1)$ and utilizing lemmas from this chapter, putting $\beta=1-\alpha-\gamma$ and using that the scaling factor $\gamma$ can be decided to be $\gamma=(1+j)\beta$ where $j=1,2.$ This gives $\gamma=2\beta$. The modulo factor $x$ in $1=x\beta$ is redundant, so we take $\gamma=2\beta$ and $ \alpha=1-3\beta$. We also employ Taylor for $hb(t, h^{2\beta}x,h^{\beta }\xi_2)= hb(t, h^{\beta}\xi_2) + \mathcal{O}(h^{1+2\beta})$ and this remainder term will be biggest as $2\beta=\gamma.$ We then get \begin{lemma} The expansion formula is, in the transversal case, reduced to \begin{align} e^{-i x_1\xi_2 /h^{1-3\beta}} (P(h)e^{i x_1 \xi_2/h^{1-3\beta}} a )(x)\\ =   h b(t, h^{\beta}\xi_2 )a
+h^{1+\beta} \xi_2 D_1 a + h^{2-2\beta} D_1D_2 a + \mathcal{O}(h^{(1+2\beta}).
\end{align} \end{lemma}
We shall use for $\alpha$ the following possibilities: $1-3\beta, \alpha$ or $2\beta$, which suit the presentation best.
We end this section by considering the important behavior of the subprincipal symbol. \begin{lemma} The subprincipal symbol $b(x, \xi)= \alpha + i\beta=b(t,h^{\beta}\xi_2)$ can be written by Taylor as \begin{equation} b(t,h^{\beta}\xi) =  b(t)+\mathcal{O}(h^{\beta}).\end{equation} We get the mapping for the imaginary part \begin{equation} t \mapsto  (\beta(t) \arrowvert_{ \Sigma_2(P(h))}+\mathcal{O}(h^{ \beta})).\end{equation} which will change sign for small enough $h$ in an interval near $0$. We then get $ D_1 i \int_0^t b_h(t)dt/\xi_2  = b_h(t) / \xi_2$ as the integrating factor for the solution. \end{lemma} \begin{proof} The change will occur for small enough $h$, and we may take $t$ to be $-t$ by choosing the sign of $\xi_2$ so that the sign changes from plus to minus. This means that $\im \int_{t_0}^t b(t)dt$ first increases on the interval and then decreases, so it has a maximum on the interval. If we subtract a constant, we may take this maximum as zero, so that $\im \int_{t_0}^t b(t)dt \le 0$ and the exponent is bounded. We start the integration of $b(t)$ at this maximum, which, after a translation, we can assume to be at $t_0 = 0$, so we get $ \int_0^t b(t) dt$. We also observe that the functions $b(t)$ are uniformly bounded in $C^{\infty}.$ \end{proof}
\section{Proof of the Theorem}
After the last section, where more than half the job was done, we are now ready for the theorem.
\begin{theorem} Let $ P(x,hD_x; h^nB_{n \geq 1}(x,hD_x))$ have a subprincipal part as an asymptotic expansion of $B(x,hD_x)$ in $h$ and a real principal symbol $p=P(x,\xi)$ that microlocally factorizes $p=p_1p_2$ at $(x_0, \xi_0) \in  \Omega$. Assume that $p^{-1}(\zeta)$ is a union of two hypersurfaces with \emph{transversal} involutive intersection at $\Sigma_2(P(h))$ and $d^2_{\xi} p(x_0, \xi_0) \not=0$ for $\zeta \in \Sigma (P(h))$. If the imaginary part of the subprincipal symbol $\beta(x,\xi)$ changes sign on a limit bicharacteristic in $\Sigma_2(\zeta) \cap \Omega$, then $\zeta \in \sigma^{\infty}_{scips}(P(h))$, the semiclassical injectivity pseudospectrum of infinite order.
\end{theorem} We notice from the proof that the approximate solution $u(h) \in L^2$ to $ P(h)u(h) \sim 0$, is supported in the neighborhood $\Omega$ where the normal form $p(x, \xi) |{\Omega}$ exists.
The Theorem follows by modus ponens from the proposition below, due to the microlocal conditions that allow us to use a normal form operator in the proof.
\begin{proposition} Let the condition be as in the theorem with the normal form \begin{equation}p(x, \xi) |{\Omega}  = \xi_1 \xi_2 \end{equation} in the neighborhood $\Omega$ of $(x_0, \xi_0)$ and the semiclassical quantization $$  P(h)= h^2D_{x_1}D_{x_2} + h B(x,hD_x)$$ with subprincipal symbol $\sigma_{sub} P(h))\arrowvert_{ \Omega} = b $.
If the imaginary part of the subprincipal symbol, $\im b=\beta, $ changes sign on a limit bicharacteristic in $\Sigma_2(z) \cap \Omega$, then
\begin{align} (N) [u(h,b)] [a_{j \geq 0}] ( ||(P(h) u(h,b) || \le  C_N h^{N}, N \ge 0,\\ u(h,b) \in L^2, ||u(h,b)|| =1; N \in \mathbb{N}).
\end{align} \end{proposition} \begin{proof} In the proof, we start to find approximate solutions to the equation $$ P(h)v_h=P(h)(e^{ix_2\xi_2/h^{1-3\beta}}a(x))=0.$$ This is done with a series of transport equations specially outlined here to suit our purpose. We use the adapted expansion formula in the lemma from the last section and factor out $h^{1+\beta}$ \begin{align} h^{1+\beta}(h^{-\beta} b(t,h^{\beta}\xi_2 )a + \xi_2 D_1 a + h^{1-3\beta} D_1D_2 a + \mathcal{O}(h^{\beta}). \end{align} If we solve the first two terms, we get (in this way, the factorization works as a new scaling; we see that in fact the $D_1$ derivative is now unlimited depending on $-\beta$) \begin{equation}D_1 a = - h^{-\beta}b_h(t)a{\xi_2}^{-1} \end{equation} and this gives for the third line \begin{equation}h^{1-3\beta} D_1D_2 a = -h^{1-3\beta} D_2(h^{-\beta}b_h(t)a {\xi_2}^{-1}) = \mathcal{O}(h^{1-4 \beta}). \end{equation} So this term enters later in the equation system. We obtain the solution \begin{equation}a = e^{-i \int_0^t b_h(t)dt/\xi_2h ^{\beta}} \end{equation} where $ D_1 i \int_0^t b_h(t)dt/\xi_2  = b_h(t)/\xi_2,$ $a_0=a$ and $a_j(t,x) = \phi_j(x)e^{-i \int_{0}^t b_h(t)dt/\xi_2 h^{\beta}}$ in the  asymptotic expansion.
We summarize and now we use $\alpha=1-3\beta=2\beta$ as we may put $\alpha=\gamma$
\begin{multline}e^{-i g/h^{2\beta}} P(x,hD)e^{i g/h^{2\beta}} a = \\ h^{1+\beta}(h^{-\beta} b(t,h^{\beta} \xi_2)a +\xi_2D_1a + h^{2\beta} D_1D_2a + \mathcal{O}(h^{\beta})) + h^{1+\beta}\mathcal{O}(h^{2\beta}),
\end{multline} and here we show both the remainder from Taylor’s series and the asymptotic expansion.
Lemma 2.12 will give the conditions for \begin{equation}a = e^{-i\int_0^t (b(t)+\mathcal{O}(h^{ \beta}) dt/\xi_2h ^{\beta}}\sim e^{-i/\xi_2  h ^{\beta} \int_0^t b(t) dt}. \end{equation} We now solve in the modulo term $ \mathcal{O}(h^{\beta})$ \footnote {When we started, we used $\beta$ as a unit factor to keep the presentation more general; however, now we switch to the modulo system to keep the results simple.} as this is the smallest number that gives the biggest residual terms. \begin{align} h^{1+\beta }(\xi_2 a + h^{-\beta}b(t)a \nonumber\\ + h^{2\beta} D_1D_2 a + \mathcal{O}(h^{\beta})). \end{align} We use our results on the asymptotic expansion $$v_h = e^{ix_2\xi_2 h^{2\beta}} \sum_{j \ge 0}  a_j(t) h^{\beta}$$ and we find \begin{equation}v_h = e^{ix_2\xi_2 h^{\alpha}}  e^{-i \int b(t)dt / \xi_2 h ^{\beta}} \sum_{j \ge 0} \phi_j(x) h^{j \beta}, \  \phi \in C_0^{\infty}(\mathbb{R} \times \mathbb{R}^n), \ \phi(0)=1. \end{equation} We use cut off functions $\phi \in C_0^{\infty}$ for $x$ so $a_j(t,x) =  e^{-i \int_{0}^t b(t)dt/\xi_2  h^{\beta}} \phi_j(x)$ and we write this as \begin{equation}P(h)v_h \sim e^{ix_2\xi_2 h^{2\beta}} h^{1+\beta} \sum_{j \ge 0} c_j h^{j \beta}.
\end{equation} We have \begin{align}c_j \sim \xi_2 D_1a_j + h^{-\beta}b(t) a_j + S_j(h) \nonumber\ \sim e^{-i \int_{0}^t b(t)dt /\xi_2 h ^{\beta}} (\xi_2D_1 \phi_j(x) + R_j(h)) = 0. \end{align} The terms $S_j(h)$ all have the exponential, so we can factor it out. The $$R_j(h) \in(h^{1+j_{1|2}\beta})\mathcal{O}(h^{\lambda+ \mu -1 + (\kappa- (j_{1|2}+1)\mu-j_{1|2})\beta }), \ j=1 $$ is the type of term that is left and it is just derivations of $\phi_k$ uniformly bounded depending only on $\phi_k$ for $k<j$.
The term $h^{2\beta}D_1D_2a$ will be entered in the second transport equation in the expansion. The system looks like $$(0) \quad \xi_2 D_1  a_0 +  h^{ -\beta}b(t)a_0=0$$ $$(1) \quad \xi_2 D_1h^{\beta}a_1 + h^{-\beta}b(t) h^{\beta}a_1 = h^{2 \beta}D_2(h^{-\beta}b(t)a_1/ \xi_2) $$ $$(2)\quad \ldots $$ $$(j)\quad \xi_2 D_1h^{j \beta}a_{j+1} + h^{-\beta}b(t) h^{j \beta}a_{j+1}= - S_na_j$$ $$\Leftrightarrow e^{-i \int b(t)dt /\xi_2 h ^{\beta}} (\xi_2D_1 \phi_j(x) + R_n(h)) = 0.$$
This is as before, and now to get it right, we also have to use a cut-off in the places where the unscaled coordinate is to get \begin{equation} (N)[h](|\text{exp}(-i \int b(t)dt /h^{\beta})| < C_N h^{N}, a_j = \mathcal{O}(h^{N}); N \in \mathbb{N}, 0<h \leq 1). \end{equation} Instead of $u(h,b)$, we have $v_h(b)$, and now we convert to $u(h,b)$ using the lemma. As the statement below is true for any $N$, we can replace $N$ by $N + (n(2\beta) +\beta)/2=N + (2n+1)\beta/2$. For the proof to be complete, we construct $u(h)$ with $||u(h)|| = 1$ in the $L^2$-norm. If we put  $u(h)=v_h/ ||v_h||$ we get $||u(h)|| =1$, and $$||P(h) u(h)|| \le c_1 h^{N + (2n+1)\beta/2} \ / || v_h||.$$  Now we use the lower estimate from Lemma 2.6 for $v_h$  together with the semiclassical limit \begin{equation}||P(h)u(h)|| \le (c_1/c_2) h^{N} = C h^{N}. \end{equation} \begin{align} \Vert R(h) \Vert  \ge  \frac{\Vert u(h) \Vert}{\Vert P(h)u(h) \Vert} \ge \frac{1}{C h^{N}}\to \infty, \  h \rightarrow 0. \end{align} We can now write the complete \emph{protocol} of the statement and the conditions for the resolvent norm in the \emph{transversal case} with the \emph{$\beta-$condition} \begin{align} (u(h,b))[b][h][P(h)][\dot\zeta][\dot\xi_1, \dot\xi_2](||R(h,\zeta)||=  \frac{\Vert u(h,b) \Vert}{\Vert (P(h)-\zeta)u(h,b) \Vert_{L_2}} \rightarrow \infty; \\ 0 <(u(h,b))\in C^{\infty}, ||u(h,b)|| =1, \im{b(t)}=\beta(t) \mapsto \pm t, h \rightarrow 0, ; \\ P(h)=h^2D_1D_2+hB(x,hD_x), \zeta=0, \textbf{.}p=dp=0 \land \textbf{.}\xi_1=\xi_2=0).\end{align} \end{proof}

\section{Factorization annihilating the $\beta$-condition}
To prepare for the next article, we shall now give an example annihilating the $\beta$-condition due to factorization. This serves as a type of conjugation, allowing us to ignore lower-order terms.
We start by considering operators that have a derivative after the subprincipal symbol, which in this case means that we can factorize
\begin{align} P(h)=h^2D_1D_2+hB(x,hD_x)hD_2;\ (x=x_1 | x_2) ; B(x,hD_x)hD_2 \in \Psi^2 \\ P(h)=h^2(D_1+ B(x,hD_x))D_2=h^2P_1P_2. \end{align} If we take $B(x,hD_x) \in \Psi^0$ for example $B=-ix$ we get \begin{equation} h^2P=h^2(D_1-ix)D_2=h^2D_2A_- .\end{equation} Here we meet the creation operator $A_+=(D_1 + ix_1)$ used in particle physics [12]. Together with annihilation $A_-=(D_1 - ix_1)$ we have \begin{equation} A_+=A_-^*. \end{equation} We note that \begin{equation} v(x)=e^{-(x_1^2+x^2_2)/2} \in S(\mathbb{R}^2) \end{equation}
is an eigenmode for $z=0$  as \begin{align} e^{(x_1^2+x_2^2)/2}Pv=-(e^{(x_1^2+x_2^2)/2}(\partial_{x_2}(\partial_{x_1}(e^{-(x_1^2+x_2^2)/2)})+ x_1\partial_{x_2}(e^{ (x_1^2+x_2^2)/2)}))\\=-(x_1x_2-x_1x_2)=0. \end{align}
This means that the resolvent $R_h(0)=(P_h)^{-1}$  does not exist as $z=0$ is an eigenvalue, and we are not able to create quasimodes for $P(h)v=0$ for $z=0$. We now construct the transport equations to solve. We do not need the earlier parameters because there are no remainder terms here, and we shall approach $z=0$ by letting $h \to 0$. \begin{equation} v_h=\text{exp}(x;h)\phi_{\Sigma j}(x;h) = e^{ix_2/ h}  e^{- x_1^2/2} \sum_{j \ge 0}   \phi_j(x_1,x_2) h^{j}. \end{equation} We choose $\zeta \in |z|=1$, and then we use the operator $P-h\zeta$ and let $h$ go to zero. In this way, we will be able to see how the resolvent norm changes as we approach the eigenvalue $z=0$. The $h$-quantification is here (with conjugation $\text{exp}^{-1}(t,x;h)$, and $(t,x)=(x_1,x_2)$ and $A_t^-=A(t)-)$ \begin{align}
\text{exp}^{-1}(hD_2(A_t^-)\text{exp}\phi_{\Sigma j}-h\zeta \phi_{\Sigma j}
\end{align}
We calculate the derivatives \begin{align} \text{exp}^{-1}(hD_2(A_t^-)(\text{exp}\phi_{\Sigma j})=\text{exp}^{-1}hD_2(A_t^-(\text{exp})\phi_{\Sigma j}+\text{exp}A_t^-\phi_{\Sigma j})\\
\text{exp}^{-1}D_2(\text{exp} A_t^-\phi_{\Sigma j})); \quad [A_t^-(\text{exp})=0]\\=\text{exp}^{-1}(D_2(\text{exp}) A_t^-\phi_{\Sigma j})+(\text{exp})D_2(A_t^-\phi_{\Sigma j}))\\ =\frac{i}{h} A_t^-\phi_{\Sigma j}+A_t^-D_2\phi_{\Sigma j} \end{align}
Finally we have \begin{equation} \sum_{j \ge 0}(i A_t^-\phi_j h^{j}+A_t^-D_2\phi_j h^{j+1}- \zeta \phi_j h^{j+2} ) \end{equation} or \begin{equation} \sum_{j \ge 0}i(D_1-it)\phi_j h^{j}+(D_1-it)D_2\phi_j h^{j+1}- \zeta \phi_j h^{j+2} ) \end{equation}
The first transport equation is $$A_-(x_1)\phi_1=0 ; \quad \phi_1= -e^{x_1^2/2}$$ as  $A_-(x_1)$ kills such terms. The second equation is zero for $\phi_1=e^{-(x_1^2+x^2)/2} \in S(\mathbb{R}^2)$ which we saw before. This leaves us with $\zeta \phi_j h^{j+2}$ which does not give any quasimodes, no $\beta$ -condition.
In [2], where we look into more generalized problems, we will come back to the factorization of the $\beta$-condition.

\appendix

\section{Preliminaries}
Here, we define some properties of the semiclassical and spectral environments in which we are working. The smooth manifold $\Sigma$ of $T^*X$, where $X$ is an open set in $\mathbb{R}^n$ and $T^*X$ is the cotangent space, is called $\emph{involutive}$ \footnote{This we can use if we like to apply the preparation theorem by Malgrange, we then get $f(x,t)=c(t^2+a_1t+a_0),$ which we can factorize taking $a_0=0$, to $ct(t+a_1).$} if it has the following  property:
For all smooth functions $u,v$ on $T^*X$, $$\{u,v\}=\sum_1^N \partial_{\xi_j}u \partial_{x_j}v - \partial_{x_j}u \partial_{\xi_j}v=0.$$
We shall consider the non radial involutive manifold $\Sigma$ with the neighborhood $\Omega \ni(x_0, \xi_0)$ where the normal form $p=p_1p_2 |_{\Omega} $ exists. We think of a given function \begin{equation} a:\mathbb{R}^n \times \mathbb{R}^n \longrightarrow \mathbb{C} \quad   \text{defined by} \quad a(x, \xi)\in C^\infty, \end{equation} where $x$ is position and $\xi$ is momentum, as a classical observable on phase space, and we call this function a $\emph{symbol}$. The total symbol $a(x, \xi)$ is often called \emph{classical} if it is a sum of terms, and then the first term is the principal symbol, $p(x,\xi)$, and the next term is the subprincipal symbol, here called $b(x,\xi)$. These two leading terms are of particular interest because they determine much of the behavior of the quantized (next paragraph) operator and have other special features, such as invariance under coordinate changes, which is very useful. Here, more complicated symbols, not only polynomials, are allowed, for example, the symbol $\frac{1}{\xi^2+1}$ which is quantized as the operator $(\Delta+1)^{-1}$. Then we call the operator \emph{pseudodifferential}. We must then impose a limit on the derivative of the symbol to make the integral (below) converge. This leads to different symbol classes; here, we use only the symbol class $S^n$ and the operator class $\Psi^n$, since we use simpler normal-form operators.
To the symbol $a(x,\xi)$ we associate a quantum observable or a pseudodifferential operator $P(x,hD)$, acting on functions $u=u(x)$ \begin{equation} P(x,hD_x)u(x) = \frac{1}{(2\pi h)^n} \iint e^{\frac{i}{h}\langle x-y,\xi\rangle}a(x,\xi)u(y)dyd\xi$$ $$= \frac{1}{(2\pi h)^n} \int e^{\frac{i}{h} x \xi} a(x,\xi) \hat{u}(\xi) d\xi\end{equation} where $h \in (0,1]$ and $D_x=\frac{\partial}{i \partial x}$. This we call the semiclassical quantization. In the case where we have $a(x,\xi)=x^{\alpha}\xi^{\alpha}$ the quantization is simple: we get operators of the type $P(x,hD_x)u= x^{\alpha}(hD_x)^{\alpha}u$, the association is $(x_j,\xi_j) \mapsto (x_j, hD_j)$ and in general the small parameter $h=1/|\xi| \in \mathbb{R}^+.$ For more on this, see Zworski [12].
We define some properties we need: The closure of the range is \begin{equation} \Sigma (P(h))= \Sigma(p) = \overline{\text{Ran}} (p) \end{equation} where $\text{Ran}(p)$ is the image of the principal symbol $p=p(x, \xi)$.
Instead of the characteristic sets $\Sigma_1(p) = \{(x, \xi):p(x, \xi) = 0; dp \not = 0\}$ and $\Sigma_2(p) = \{(x,\xi):p(x,\xi) = 0; dp  = 0\}$ for classical operators, we here look at level sets \begin{equation} S=p^{-1}(z) = \{(x_0, \xi_0): p(x_0, \xi_0) = z\} \end{equation} for $z \in \Sigma (P(h)).$ The stationary points are \begin{equation} \Sigma_2(P(h)) = \{d_{\xi}p(x_0, \xi_0) = 0\} \end{equation} and we define the $\emph{stationary point level set}, \Sigma_2(\zeta) =  p^{-1}(z) \cap \Sigma_2(P(h))$ as \begin{equation} [x_0] [\xi_0](p(x_0, \xi_0)= \zeta,d_{\xi}p(x_0, \xi_0) = 0,\zeta \in \Sigma (P(h))). \end{equation}
A symplectic \footnote {From Greek: $\pi \lambda \epsilon \kappa \tau \omega$ to knit and $\sigma \upsilon \mu$ which means to include, together the words could mean \emph{entangle}.} space is modeled after classical mechanics where we follow an object in space over time $(t)$ for motion $x(t)$ and momentum $\xi(t)=mv$ with respect to speed$ (v)$ and mass $(m).$
In this space, there is a non-degenerate (non-characteristic) form $\omega= dx \land d\xi$ which can be integrated $\int \omega$ to show an area. The size of this area is invariant; it can be distorted, but it cannot shrink or get larger.
By a symplectic change of coordinates, we mean canonical transformations that keep the qualities of the operator intact. For example, the Hamilton vector flow $(x(t),\xi(t))= \text{exp}$ $(tH_p)$ is symplectic, which means that it is invariant, not depending on any particular system of coordinates. The Poisson bracket $\{u,v\}$ and the subprincipal symbol on $\sigma_{\text{sub}}(P(h))\arrowvert_{\Omega} = b(x, \xi) $  are also invariant, for more of this see [6].
\section{The Spectrum and the Psedospectrum }
We now turn to our introduction of pseudospectrum and the spectrum of a classical operator $P$ in $L^2$. This is quite general, aimed at readers who are not so familiar with this notion.
Let $H$ be a complex Hilbert space. In $P(h)$, the family of operators depending on $h$, we fix $h_0 \in(0,1]$ and define $P_h:=P(h_0).$ \begin{definition} Here, $ T: H \supset D(T) \to H$ is a closed, densely defined operator: $\overline {D(T)} = H$, and we write $T \in \mathscr{C}(X)$ for this operator. It has the property that every sequence $$(u_k)[Tu_k](({u_k}\to u,{Tu_k}\to v) \land (u \in D(T), Tu=v); u \in H. )$$ A bounded operator $ E \in \mathscr{B}(X)$ satisfy the expression $$ (x)[k](|E(x)| \le k |x|) \ \land \ E\in C^0 \ \land \  E \ \text{is uniformly} \ C^0.$$ \end{definition} See Kato [7], which we follow here, for an expository introduction to operator theory. \begin{remark} The boundedness of a linear operator is not the same as, for example, the mapping $f:\mathbb{R} \rightarrow \mathbb{R}$. The operator $T: X \rightarrow X,$ defined by $T(x)=x$ is bounded as an operator but not in the sense of a bounded function. When  $h\rightarrow 0$, the symbol $\xi_1 \xi_2$ is bounded as we have $h \sim 1/\xi $. This is one of the main ideas in semiclassical analysis. We will not encounter unlimited derivatives more than $hD_1$, which is unlimited because of a factorization of the expression where it belongs. \end{remark} \begin{definition} \begin{align} [\zeta][\mathcal{R}](\mathcal{R}(\zeta)=(P_h-\zeta)^{-1},\mathcal{R}\in C^0, \mathcal{R}\in \mathscr{B}(X),\zeta \in \mathbb{C}) \end{align} is called the resolvent of $P_h$. (We use Greek $\zeta$ for those values when the resolvent exists. The number $\zeta$ in \begin{equation} [\zeta](\mathcal{R}(\zeta)=(P_h-\zeta)^{-1};\zeta \in \mathbb{C}) \end{equation} is called a \emph{regular value} for $\mathcal{R}$. \end{definition} We define the spectrum and the pseudospectrum in the standard way used, for example, in [10], and use $\sigma_{s}$ to denote the spectrum, and index this letter further to denote other forms of this notion (and $\lnot [\cdot]:=$ does not exist). \begin{definition}The spectrum for $P_h$, denoted $\sigma_{s}(P_h),$ is defined by \begin{equation} (P_h)[z](\lnot[(P_h-z)^{-1}]; z \in \mathbb{C}) \end{equation} The resolvent set $\rho(P)$ are all the regular values \begin{equation} (P)[\zeta]( [(P_h-\zeta)^{-1}], \ \zeta \in \mathbb{C}) \end{equation} and it is thus the complement of the spectrum $\rho(P_h)=(\sigma_{s}(P_h))^c  .$
\end{definition} \begin{remark} Other definitions of the pseudospectrum has been used e.g. $(\Lambda(p))$ which is the closure of the set below and $m\in T^*\mathbb{R}^n$(phase space)
\begin{equation} \Lambda(p):=p(\{p,\bar{p}\}(m)\not=0).
\end{equation} Here, $$\{p,\bar{p}\}= \sum \partial_{\xi_j}p ; \partial_{x_j}\bar{p} - \partial_{x_j}p ; \partial_{\xi_j}\bar{p}$$ is the Poisson bracket. Since the principal symbol has zero derivatives, $dp=0$, this does not work here; the bracket is zero. Different \emph{twist} conditions can also be found in TRE[10] such as $$\im(\partial_x a) / \partial_{\xi}a)=C_0>0 \Leftrightarrow \im(\partial_x a- C_0  \partial_{\xi}a)>0.$$ This derivative concerns the total symbol $a=a(x,\xi)$, so it is not automatically zero. For example, the operator $h^2D^2+ix^2$ has full symbol $a= \xi^2+ix^2$ and twist condition $x/\xi$ which would not be defined for just $p=\xi^2$. In the same book, we can also find the theory of winding number, which gives a geometric answer to the question of existence. However, it is stated that the results (Theorem 11.2, page 105, in [10]) only apply in one dimension. \end{remark} \begin{remark} Here we reduce the resolvent operator $R(h,\zeta)$ to the form $R(h)= P(h)^{-1}$ by subtraction. $R(h)$ is thus just the inverse of the operator $P(h)$. This is possible because we can, from the start, study the operator $Q(h) = (P(h) + \zeta)$.  When we look at $Q(h)-\zeta=((P(h)+\zeta)-\zeta))=P(h)$ the result will be the same with our without the constant $\zeta$, which is just a translation, f.x. in $(x^2 \sim x^2+a)$ or $(D_x(x^2)= D_x(x^2+a))$.
\end{remark} \begin{definition} We let $P_h \in \mathscr{C}(X)$ and $\epsilon$ > 0 and we say that the $\epsilon$-pseudospectrum  $\sigma_{\epsilon ps}(P_h)$ is the set of $\zeta \in \mathbb{C}$ in the mapping
$$\sigma_{\epsilon ps}:\Vert \mathcal{R}(\zeta) \Vert \longrightarrow \mathbb{R}^+$$
if we have $$ (1) \qquad (\Vert(P_h - \zeta)^{-1} \Vert > \epsilon^{-1} \   \lor   \  \Vert P_h^{-1} \Vert > \epsilon^{-1}). $$ For the function $u$ with $\Vert u \Vert =1$  the following are equivalent conditions
$$ (2) \qquad( \Vert (P_h-\zeta)u \Vert < \epsilon \   \lor   \  \Vert P_hu \Vert < \epsilon ).$$
Another view of this notion we find  in
$$ (3) \qquad [\dot E] (z \in \sigma_{s}(P_h+E),  E \in \mathscr{B}(X), \Vert E \Vert < \epsilon).$$ \end{definition}
The form on the right in (2) is of the type used in our proofs (but we, of course, use the semiclassical type $|| P(h)u(h)||_{L^2}$). The standard definition is used when performing numerical calculations.
In (2)  $\zeta$ is an $\epsilon$-pseudo eigenvalue of $P_h$ and $\upsilon$ is the corresponding $\epsilon$-pseudo eigenvector or pseudo-/quasimode which we will work with.
\begin{proposition} Let the conditions be as above. If $ \Vert P_hu \Vert < \epsilon$ for some $ u $  with  $\Vert u \Vert =1$, then $ \Vert P_h^{-1} \Vert > \epsilon^{-1}$ so we have $(2)\Rightarrow (1)$. \end{proposition} \begin{proof}Let $u$ be a quasimode for the equation $P_h u \sim 0$ so that $\Vert P_hu \Vert < \epsilon$. Then, for some positive numbers $\delta < \epsilon$, we have $$ P_hu = \delta v $$ with $u$ and $v$ both having norm one. We get, if $P_h$ is invertible, \begin{equation}P_h^{-1}v = \delta^{-1} u,  \Vert R_h\Vert =  \Vert P_h^{-1} \Vert \Vert v \Vert > \Vert  \delta^{-1} u \Vert = \vert \delta^{-1}\vert \Vert v \Vert = \delta^{-1}> \epsilon^{-1}.\end{equation} \end{proof} Now, if we look at the third definition, we find that the number $z$ is in the $\epsilon$-pseudospectrum if it belongs to the spectrum of some perturbed operator $P_h+E$. If we want to compute the eigenvalues numerically, the starting point is a discretization of the operator. Here we can get round-off errors and perturbations of the initial operator. When we compute, the algorithms used to calculate the eigenvalues determine the eigenvalues of a perturbed operator. The result is that the spectrum can be unstable; the pseudospectrum is, in fact, the spectrum of a small perturbation of the operator. \begin{proposition} Let the conditions be as above. If $ \Vert P_hu \Vert < \epsilon$ for some $ u $  with  $\Vert u \Vert =1$   then $z \in \sigma_{s}(P_h+E) \text{ for some }  E \in \mathscr{B}(X) \text{ with } \Vert E \Vert < \epsilon$ so we have $(2)\Rightarrow (3)$. \end{proposition} \begin{proof} We start as above, so let $u$ be a quasimode for the equation $P_h u \sim 0$ so that $\Vert P_hu \Vert < \epsilon$. Then, for some positive numbers $\delta < \epsilon$, we have $$ P_hu = \delta v u $$ with $u$ and $v$ both having norm one. This gives $P_hu - \delta vu=0$ and we may use $-\delta v$ as a bounded operator $E$ to write $P_hu+Eu=(P_h+E)u=0$. This means that $P_h+E$ is not invertible for $\zeta=0$, so  it is now in the spectrum. \end{proof} \begin{remark} The definition of semiclassical injectivity pseudospectrum can also be used the other way around as a $\emph{a priori estimate}$ \begin{equation} ||u(h)|| = 1 \text{and} \Vert P(h)u \Vert > Ch^{N}. \end{equation} This means we have no quasimodes to expect. If we consider the one-dimensional harmonic oscillator $P(h)=h^2 D_x +x^2$ we find, with $\langle x,x \rangle= |x|^2,$ $$\Vert P(h)u \Vert^2=\langle P(h)u,u \rangle= \langle h^2D_1^2u,u\rangle +\langle x_1^2,u\rangle=|h D_1u|^2+ |xu|^2 \ge 2|h D_1u| |xu| $$ $$ \ge   |2 \im\langle hD_u,xu\rangle |= |\langle xu, hD_1u \rangle- \langle hD_1u,xu \rangle | \quad (2 \im z =z- \bar z) $$ $$=|\langle  hD_1xu,u \rangle- \langle xhD_1u,u \rangle | \quad (\langle Tx,x\rangle= \langle x,T^*x \rangle)$$ $$=\langle (hD_1x-xhD_1)u,u\rangle $$ $$= |\langle [hD_1,x]u,u \rangle| \quad ([u,v]=uv-vu)$$ $$|u|^2; \quad (hD_1xu-xhD_1u)=hu/i,$$ so that in $\Vert P(h)u \Vert > Ch^{N}$ we get $C=N=1.$ \end{remark}
\section{Terminology and Notations}
We use here some unorthodox (and orthodox) notations, for example, the quantifiers we denote as done in older books in logic, as WR(Principia Mathematica)[12] $(x)$ for all $x$. Inspired by this, we have taken $[x]$ as “there exists” $x$ (this symbol does remind us of $\exists $) and $[\dot x]$ “there exists a unique” $x$. The idea is to avoid symbols that might steal our attention; for example, the use of \% is avoided in text; it is spelled out as ‘percent’ in literature, following an old writing codex used by journalists and writers.
\begin{example} Another example of this is $\exists !$, which is a mix of an unknown letter $\exists$ and the known sign !. This is taken from a mathematical text.
$$(\exists!x: P(x)) \Leftrightarrow \exists x:(P(x) \land \forall y:(P(y) \Rightarrow (y = x)))$$
This we write: $$([\dot x]P(x):= [x](P(x) \land (y)(P(y)) \land (y=x)).$$ Besides smoother symbols, it is also clear what the definiendum, to the left, is, and we note that the definition is conservative; it uses nothing more than what is known (definiens) to the right. The definition uses a double conjunct, which is easier to catch at a glance, giving the definition more structure. \end{example} Parentheses of the kind we use here are more accustomed to the eye. If there are too many, we can use dots “$\textbf{.}$”, used solely in WR, but we use them together with parentheses and only in an obvious way. \footnote{There is also the so-called Polish notation that has no parentheses, but it has not gained entry among mathematicians due to its being apart from the usual way of writing, although this prefix notation  has great advantages and in simple cases are obvious as 5(3+4) $\mapsto (\times+345$), working from the inside out.}
Quantifiers ($(x), [y],$ and $[\dot x])$, statements $(S_n)$ and conditions $(c_n)$ are ordered as follows \begin{align}(x)(y)\ldots[x][y]\ldots[\dot x][\dot y]\ldots(S_1,S_2, \ldots; c_1,c_2 \ldots ;\ \text{bounded}\ x,y, \ldots ; \text{free}\ x,y, \ldots)
\end{align} and (;) are used to divide information when needed to help the reading.  Variables $(x,y,\ldots)$ are called \emph{bounded} if they are in the quantifiers; others are called “free”. \begin{example} We connect to the continuity  of  the operator-valued function  $(h,\zeta)\mapsto (P-\zeta)^{-1}$ which we study here, or more correctly, the norm of this function.
Point wise continuity in this case  $$(\delta)(\zeta)[\epsilon][\dot\zeta_0](|\zeta-\zeta_{0}|<\delta \Rightarrow |\mathcal{R}(\zeta)-\mathcal{R}(\zeta_0))|<\epsilon;\   \delta>0,\epsilon>0;\ \delta,\epsilon \in \mathbb{R}^+, \zeta, \zeta_0\in \mathbb{C})$$ and for uniformly continuity $$(\delta)(\zeta)(\omega)[\epsilon](|\zeta-\omega|<\delta \Rightarrow |\mathcal{R}(\zeta)-\mathcal{R}(\omega))|<\epsilon;\ \delta>0,\epsilon>0; \delta,\epsilon \in \mathbb{R}^+, \zeta, \omega\in \mathbb{C})$$ \end{example} The conditions or statements in the parentheses, which substitute for “such that”, are all in a conjunctive state, symbolized by the comma (,); if not, we use connectives  $\land$ (and) and $\lor$ (or) between them.
Of the Greek letters, we use $\alpha, \beta$ and $\gamma$ as parameters in the system of transport equations. They are positive numbers less than 1, and they will fulfill
$\alpha+\beta +\gamma=1.$ The letters $\alpha$ and $\beta$ are also used for the complex symbol $b=\alpha(x,\xi) + i\beta(x,\xi)$. What is often clear from the context, otherwise we use $\alpha(x,\xi)$ or $i\beta(x,\xi)$. We use $\lambda$ and $\mu$ to denote powers of derivatives as in $h^{\lambda+\mu(1-\gamma) }D_1^{\lambda}hD_j^{\mu}$, and as usually $hD_x=\frac{1}{i}\partial_x.$ The letter $0<h \ll 1$ denotes the semiclassical parameter that goes to zero in the semiclassical limit. Often we also write $0<h \leq 1$ as we can use rescaling to the case $h=1.$ The letter $\sigma$ is used to denote different forms of spectrum and pseudospectrum, but is also used to indicate the symbol of an operator, as in $\sigma(P(h)$, where $P(h)$ is a semiclassical operator.
We observe that denotation is a relation $\mathcal{R}: a(a_1,a_2,\ldots) \ (R) \ b(b_1,b_2,\ldots)$, not always a function, $f$, so that we can have $a_1 \mapsto (b_1,b_2,\ldots)$, but we have $f \subset \mathcal{R}.$
For the coordinates we use $(t,x)=(x_1,x’)$ and $\tau,\xi)=(\xi_1,\xi’).$

\end{document}